\documentclass[12pt,a4paper]{article}
\usepackage[utf8]{inputenc}
\usepackage{amsmath}
\usepackage{amsfonts}
\usepackage{amssymb}
\usepackage{graphicx}
\usepackage{amsthm}
\usepackage{enumerate}
\usepackage{setspace}
\usepackage{booktabs}
\usepackage{color}

\newtheorem{theorem}{Theorem}[section]
\newtheorem{problem}{Problem}[section]
\newtheorem{lemma}{Lemma}[section]

\newtheorem{proposition}{Proposition}[section]
\newtheorem{conjecture}{Conjecture}[section]

\newtheorem{definition}{Definition}[section]

\usepackage[top=1.42in, bottom=1.42in, left=1in, right=1in]{geometry}

\newcommand{\mc}[1]{\mathcal{#1}}
\newcommand{\K}{\mathcal{K}}
\newcommand{\ds}{\leq}
\newcommand\numberthis{\addtocounter{equation}{1}\tag{\theequation}}

\author{Thomas Perrett\footnote{Research supported by ERC Advanced Grant GRACOL, project number 320812.}\\
Department of Applied Mathematics and Computer Science,\\
Technical University of Denmark,\\ DK-2800 Lyngby, Denmark\\
\texttt{tper@dtu.dk}}
\title{Chromatic roots and minor-closed families of graphs}
\begin{document}
\maketitle

\begin{abstract}
Given a minor-closed class of graphs $\mc{G}$, what is the infimum of the non-trivial roots of the chromatic polynomial of $G \in \mc{G}$? When $\mc{G}$ is the class of all graphs, the answer is known to be $32/27$. We answer this question exactly for three minor-closed classes of graphs. Furthermore, we conjecture precisely when the value is larger than $32/27$.
\end{abstract}
%
\begin{section}{Introduction}
The \emph{chromatic polynomial} $P(G,t)$ of a graph $G$ is a polynomial which counts, for each non-negative integer $t$, the number of proper $t$-colourings of $G$. It was introduced by Birkhoff~\cite{chrompolyIntroduced} in 1912 for planar graphs and extended to all graphs by Whitney~\cite{Whitney2,Whitney1} in 1932. More recently, several results have been obtained on the distribution of the real and complex roots of the chromatic polynomial, see for example the survey article~\cite{JacksonSurvey}.

We say that a real number $t$ is a \emph{chromatic root} of a graph $G$ if $P(G,t)= 0$. Since $0$ and $1$ are chromatic roots of any graph with at least one edge, we say these are \emph{trivial}. Tutte~\cite{Tutte01} proved that the intervals $(-\infty ,0)$ and $(0,1)$ contain no chromatic root of any graph, so all non-trivial chromatic roots are greater than $1$. 

For a class of graphs $\mc{G}$, define $\omega(\mc{G})$ to be the infimum of the non-trivial chromatic roots of $G \in \mc{G}$. We define $\omega({\mc{G}}) =2$ if the infimum does not exist. Thus for a class of graphs $\mc{G}$, the interval $(1, \omega(\mc{G}))$ contains no chromatic roots of $G \in \mc{G}$, and the endpoint $\omega(\mc{G})$ can be included if the infimum is not attained. Motivated by what was then the $4$-colour conjecture, Birkhoff and Lewis~\cite{BirkLewis} showed that $\omega(\mc{G}) = 2$ when $\mc{G}$ is the class of planar triangulations. However, since bipartite graphs with an odd number of vertices have a chromatic root in $(1,2)$, the problem of determining $\omega(\mc{G})$ when $\mc{G}$ consists of all graphs was open until Jackson proved the following surprising result.

\begin{theorem}\emph{\cite{Jackson32/27}}\label{thm:32/23}
If $\mc{G}$ is the class of all graphs, then $\omega(\mc{G}) = 32/27$.
\end{theorem} 

In the proof of Theorem~\ref{thm:32/23}, Jackson introduced a class of graphs $\K$, whose elements are called \emph{generalised triangles}, and showed that a sequence of graphs in $\K$ have chromatic roots converging to $32/27$ from above. Later, Thomassen~\cite{RootsDenseThomassen} strengthened this by showing that chromatic roots are dense in $(32/27, \infty)$. Since the graphs known to have chromatic roots close to $32/27$ have a very particular structure, it is natural to ask if $\omega(\mc{G})>32/27$ holds for restricted classes of graphs. This has been studied by several authors, for example Thomassen proved the following.

\begin{theorem}\emph{\cite{ThomassenHamPath}}\label{thm:THomassenHamPath}
If $\mc{H}$ is the class of graphs with a Hamiltonian path, then $\omega(\mc{H}) = t_0$, where $t_0 \approx 1.296$ is the unique real root of the polynomial $t^3-2 t^2+4 t-4$.
\end{theorem}

Dong and Koh suggested the problem of determining $\omega(\mc{G})$ for minor-closed classes of graphs. They proved the following theorem, which implies that one need only investigate the graphs in $\mc{G}$ which are generalised triangles.
\begin{theorem}\label{thm:DongMinor}\emph{\cite{DongKohMethod}}
If $\mc{G}$ is a minor-closed class of graphs, then $\omega(\mc{G}) = \omega(\mc{G} \cap \mc{K})$.
\end{theorem}

\begin{figure}
\centering
\begin{tabular}{cc}
\toprule
Forbidden minor & $\omega(\mc{G})$\\ 
\midrule
$K_3$ & $2$ \\ 
$K_4 - e$ & $2$ \\ 
$K_{2,3}$ & $2$ \\ 
$K_4$ & $32/27$ \\ 
$K_{2,4}$ & $\beta \approx 1.43$, chromatic root of $K_{2,3}$ \\
$K_{2,r}$ & $32/27+\varepsilon (r)$ \\
\bottomrule
\end{tabular} 
\caption{Results of Dong and Koh~\cite{DongKohMethod}.}\label{fig:table}

\end{figure}

Using Theorem~\ref{thm:DongMinor}, they determined the value of $\omega(\mc{G})$ for several classes of graphs characterised by forbidding a particular graph as a minor. Their results are summarised by the table in Figure~\ref{fig:table}.

In this paper we prove that certain subsets $\K' \subseteq \K$ can be considered minor-closed within $\K$, in the sense that there is a minor-closed class of graphs $\mc{G}$ such that $\mc{G}\cap \K = \K'$. Using Theorem~\ref{thm:DongMinor}, we have $\omega(\mc{G}) = \omega(\mc{G} \cap \K) = \omega(\K')$, so determining $\omega(\K')$ gives the value $\omega(\mc{G})$ for the much larger class $\mc{G}$. To illustrate this new technique, we analyse three natural subfamilies of generalised triangles and precisely determine $\omega(\mc{G})$ for three minor-closed families of graphs. 

\begin{figure}
\centering
\includegraphics[scale=0.7]{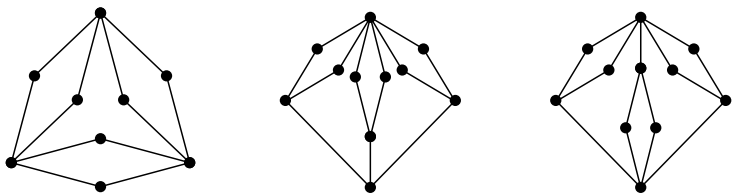}
\caption{From left to right, the graphs $H_0$, $H_1$ and $H_2$.}\label{fig:graphs}
\end{figure}
\begin{theorem}\label{thm:megaThm}
Let $H_0$, $H_1$ and $H_2$ be the graphs in Figure~\ref{fig:graphs}. 
\begin{enumerate}[(i)]
\item If $\mc{G}$ is the class of $H_0$-minor-free graphs, then $\omega(\mc{G}) = 5/4$.
\item If $\mc{G}$ is the class of $\{H_1,H_2\}$-minor-free graphs, then $\omega(\mc{G}) = q$, where $q \approx 1.225$ is the real root of $t^4-4t^3+4t^2-4t+4$ in $(1,2)$.
\item If $\mc{G}$ is the class of $\{H_0,H_1,H_2\}$-minor-free graphs, then $\omega(\mc{G}) = t_0$, where $t_0 \approx 1.296$ is the unique real root of $t^3-2 t^2+4 t-4$.
\end{enumerate}
\end{theorem}

For all previously investigated minor-closed classes $\mc{G}$, the intersection $\mc{G} \cap \K$ is either finite or equal to $\K$. In such cases it is easy to determine $\omega(\mc{G})$: If $\mc{G}\cap \K = \K$ then $\omega(\mc{G}) = 32/27$, while if $\mc{G}\cap \K$ is finite, then $\omega(\mc{G})$ is the minimum of the non-trivial roots of $G \in \mc{G} \cap \K$, which is a finite problem. In contrast to this, each class of graphs $\mc{G}$ in the statement of Theorem~\ref{thm:megaThm} has the property that $\mc{G} \cap \K$ is an infinite proper subset of $\K$. Thus, we answer Dong and Koh's question in the first three non-trivial cases. Since the graphs $H_0$, $H_1$ and $H_2$ are some of the smallest generalised triangles, our results also provide evidence for the following conjecture.

\begin{conjecture}
If $\mc{G}$ is a minor-closed class of graphs, then $\omega(\mc{G})>32/27$ if and only if $\mc{G}$ does not contain all generalised triangles. 
\end{conjecture}

Finally, the intervals we find coincide with those obtained or conjectured for other, seemingly unrelated, families of graphs. Notice for example that the interval in Theorem~\ref{thm:megaThm}(iii) is the same as that of Theorem~\ref{thm:THomassenHamPath}. This connection will be fully explained. Furthermore, the intervals in parts (i) and (ii) of Theorem~\ref{thm:megaThm} coincide precisely with those in important conjectures of Dong and Jackson. These conjectures would have implications for the chromatic roots of $3$-connected graphs, about which very little is currently known. We describe how our results suggest that it might be fruitful to attack a relaxed version of these conjectures.

The structure of the paper is as follows. In Section~\ref{sec:GTandMinors} we make several observations about generalised triangles and minors. We apply these results in Section~\ref{sec:restrictedFamilies} to obtain Theorem~\ref{thm:megaThm}. Two lengthier proofs are deferred to Section~\ref{sec:proofLemmas}.

\end{section}


\begin{section}{Generalised triangles and minors}\label{sec:GTandMinors}

All graphs in this paper are finite and simple, that is they have no loops or multiple edges. A \emph{cut-set} $S$ of a graph $G$ is a set of vertices whose removal increases the number of components of $G$. If $S = \{u\}$, then we say $u$ is a \emph{cut-vertex}. If $|S|=2$, then we refer to $S$ as a \emph{2-cut}. Suppose $G$ is a $2$-connected graph and $\{x,y\}$ is a $2$-cut of $G$. Let $C$ be a connected component of $G -\{x,y\}$ and $B = G[V(C)\cup \{x,y\}]$. We say that $B$ is an \emph{$\{x,y\}$-bridge} of $G$. If $|V(B)|=3$, then we say $B$ is \emph{trivial}.

Jackson~\cite{Jackson32/27} defined the following operation on a graph $G$ called \emph{double subdivision}: choose an edge $uv$ of $G$ and construct a new graph from $G-uv$ by adding two new vertices and joining both of them to $u$ and $v$. A \emph{generalised triangle} is either $K_3$ or any graph which can be obtained from $K_3$ by a sequence of double subdivisions. We denote the class of generalised triangles by $\K$. 

We shall require the following properties of generalised triangles which were given by Dong and Koh, see also~\cite{Jackson32/27}.

\begin{proposition}\emph{\cite{DongKohMethod}}\label{prop:Properties1a}
A graph $G$ is a generalised triangle if and only if it satisfies both of the following conditions. 
\begin{enumerate}[(i)]
\item G is $2$-connected but not $3$-connected.
\item For every $2$-cut $\{x,y\}$, we have $xy \not\in E(G)$ and there are precisely three $\{x,y\}$-bridges, none of which is $2$-connected.
\end{enumerate}
\end{proposition}

Additionally, the following observations will be useful.

\begin{proposition}\label{prop:properties2}
Suppose that $G$ is a generalised triangle, $\{x,y\}$ is a $2$-cut of $G$, and $B$ is an $\{x,y\}$-bridge of $G$. The following hold.
\begin{enumerate}[(i)]
\item $B+xy$ is a generalised triangle.
\item If $B$ is non-trivial, then there is a $2$-cut $\{u,v\} \subseteq V(B)$, such that the two $\{u,v\}$-bridges of $G$ which are contained in $B$ are trivial.
\end{enumerate}
\end{proposition}
\begin{proof}
To verify (i), it is easy to check that the conditions in Proposition~\ref{prop:Properties1a} hold for the graph $B+xy$. To prove part (ii), let $z$ be a vertex of $G$ not in $B$. Choose a $2$-cut $\{u,v\} \subseteq V(B)$ such that the $\{u,v\}$-bridge of $G$ containing $z$ has as many vertices as possible. Suppose some $\{u,v\}$-bridge $B'$ contained in $B$ is not trivial. Proposition~\ref{prop:Properties1a} implies that $B'$ has a cut-vertex $w$. Since $B'$ is not trivial, one of $\{u,w\}$ and $\{v,w\}$ is a $2$-cut of $G$ and the bridge of this $2$-cut containing $z$ is larger, a contradiction.
\end{proof}

The double subdivision operation defines a partial order on the class of generalised triangles. More precisely, for $G,H \in \K$ we define $H \ds G$ if $G$ can be obtained from $H$ by a sequence of double subdivisions. The key observation of this paper is that the minor operation gives rise to the same partial order on $\K$.

\begin{lemma}\label{lem:DSandMinorEquiv}
If $G, H \in \K$, then $H$ is a minor of $G$ if and only if $G$ can be obtained from $H$ by a sequence of double subdivisions.
\end{lemma}

\begin{proof}
If $G$ can be obtained from $H$ by a sequence of double subdivisions, then clearly $H$ is a minor of $G$. To prove the forward implication we proceed by induction on $|V(H)|$. If $|V(H)|=3$, then $H = K_3$ and the result follows from the definition of $\K$. So suppose $|V(H)|>3$ and the result holds for all generalised triangles on fewer vertices. Let $G \in \K$ such that $H$ is a minor of $G$. Since $G \in \K$, we may fix a sequence of graphs $G_0, G_1, \dots, G_r$, such that $G_0 = K_3$, $G_r = G$, and for each $i \in \{1, \dots, r\}$, $G_i$ is obtained from $G_{i-1}$ by a double subdivision operation. Let $i \in \{1, \dots, r\}$ be minimal so that $G_i$ has an $H$-minor, but $G_{i-1}$ does not. It suffices to show that $G_i$ can be obtained from $H$ by a sequence of double subdivisions.

Let $uv$ be the edge of $G_{i-1}$ which is double subdivided to form $G_i$, and let $x, y \in V(G_i)$ be the new vertices created. Also, let $B$ be the $\{u,v\}$-bridge of $G_i$ not containing $x$ or $y$. Since $G_i$ is a generalised triangle, Proposition~\ref{prop:Properties1a} implies that $B$ has a cut-vertex $w$ which separates $u$ from $v$. We let $L_u$ and $L_v$ denote the blocks of $B$ containing $u$ and $v$ respectively, see Figure~\ref{fig:proof1}. Finally, let $J$ be a fixed $H$-minor of $G_i$.

\textbf{Claim: $ux, xv, vy, yu \in E(J)$.}

Since $H$ is $2$-connected and a minor of $G_i$ but not $G_{i-1}$, the vertices $u$ and $v$ are not identified to form $J$. Furthermore, at least one of $x$ and $y$, say $x$, has neither of its adjacent edges deleted or contracted. Since $|V(H)| >3$, we have that $\{u,v\}$ is a $2$-cut of $J$ with precisely three $\{u,v\}$-bridges $B^J_1, B^J_2$ and $B^J_3$, one of which, say $B^J_1$, is the path $uxv$. It remains to show that one of $B^J_2$ or $B^J_3$ is the path $uyv$, so suppose for a contradiction that this is not the case, and that $B^J_2 \cup B^J_3$ is a minor of $B$. Every path in $B$ from $u$ to $v$ must go through $w$. However, since $B^J_2$ and $B^J_3$ are distinct $\{u,v\}$-bridges, they each contain a path from $u$ to $v$, and these paths are internally disjoint. It follows that to form $J$, the vertex $w$ must be identified with either $u$ or $v$, say $u$. In fact, since $H$ is $2$-connected, $J$ is a minor of the graph formed from $G_i$ by contracting the whole of $L_u$ to a single vertex. Thus $B^J_2\cup B^J_3$ is a minor of $L_v$, see Figure~\ref{fig:proof1}. Now let $P$ be a path from $u$ to $w$ in $L_u$. Since $P$ has at least one edge and $B^J_1$ is a trivial bridge, it follows that $B^J_1$ is a minor of the graph $P+uv$. But now $J$ is a minor of the graph $L_v \cup P + uv$, which is a subgraph of $G_{i-1}$. This contradicts the fact that $H$ is not a minor of $G_{i-1}$, completing the proof of the claim.

\begin{figure}
\centering
\includegraphics[scale=.8]{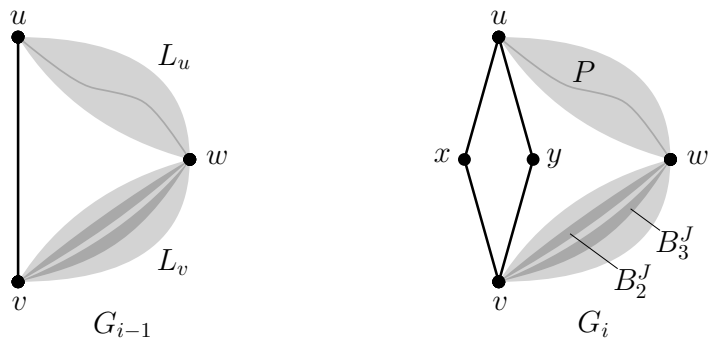}
\caption{The structure of $G_{i-1}$ and $G_{i}$ in the proof of Lemma~\ref{lem:DSandMinorEquiv}.}\label{fig:proof1}
\end{figure}

Define $J'$ to be the graph formed from $J$ by deleting $x$ and $y$ and adding the edge $uv$. Note that $J$ is formed from $J'$ by applying the double subdivision operation to $uv$ and Proposition~\ref{prop:properties2}(i) implies that $J'$ is a generalised triangle. Clearly $G_{i-1}$ contains $J'$ as a minor. By induction, $G_{i-1}$ can be formed from $J'$ by a sequence of double subdivisions. That is, there exists a sequence of graphs $J'_0, J'_1, \dots, J'_s$ such that $J'_0 = J'$, $J'_s = G_{i-1}$ and for each $i \in \{1, \dots, s\}$, $J'_i$ is obtained from $J'_{i-1}$ by a double subdivision operation. If, for some $i \in \{1, \dots, s\}$, the edge $uv \in E(J'_{i-1})$ is double subdivided to form $J'_i$, then $G_{i-1}$ contains $H$ as a minor, a contradiction. Thus the double subdivision operation is never applied to the edge $uv$. For $i \in \{0,\dots,s\}$, let $J_i$ be the graph obtained from $J'_i$ by applying the double subdivision operation to $uv$. Then $J_0 = J$, $J_s= G_i$ and for $i \in \{1, \dots, s\}$, $J_i$ is obtained from $J_{i-1}$ by a double subdivision operation. Thus $G_i$ can be obtained from $H$ by a sequence of double subdivisions as required.
\end{proof}

Let $\mc{A} \subseteq \K$. We say $\mc{A}$ is a \emph{downward-closed} subset of $\K$ if for all $G \in \mc{A}$ and $H \in \K$, we have that $H \ds G$ implies $H \in \mc{A}$. By Lemma~\ref{lem:DSandMinorEquiv}, such subsets behave as minor-closed classes within $\K$, and so have a forbidden minor characterisation within $\K$. This is made precise in the following lemma.

\begin{lemma}\label{lem:ForbMinor}
Let $G \in \K$, and suppose $\K'$ is a downward-closed subset of $\K$. If $\mc{F}=\mc{F}(\K')$ is the set of minimal elements of $\K \setminus \K'$, then $G \in \K'$ if and only if $G$ is $\mc{F}$-minor-free.
\end{lemma}

In practice, if $\K'$ is a class of generalised triangles defined by some graph property, then $\K'$ is frequently downwards-closed. It is often possible to exploit this and determine the value of $\omega(\K')$. When combined with the observations above, one obtains the value of $\omega(\mc{G})$ for a much larger class $\mc{G}$.

\begin{theorem}\label{thm:Mainthm}
Let $\K'$ be a downward-closed subset of $\K$ and let $\mc{F} = \mc{F}(\K')$ be the set of minimal elements of $\K \setminus \K'$. If $\mc{G}$ is the class of $\mc{F}$-minor-free graphs, then $\omega(\mc{G}) = \omega(\K')$.
\end{theorem}

\begin{proof}
Since $\mc{G}$ is minor-closed, Theorem~\ref{thm:DongMinor} gives that $\omega(\mc{G}) = \omega(\mc{G} \cap \mc{K})$. By Lemma~\ref{lem:ForbMinor} and the definition of $\mc{F}$, we have $\mc{G} \cap \mc{K} = \K'$, whence $\omega(\mc{G}) = \omega(\mc{G} \cap \mc{K}) = \omega(\K')$ as claimed.
\end{proof}

\end{section}


\begin{section}{Restricted families of generalised triangles}\label{sec:restrictedFamilies}

In this section we apply the the method described in Theorem~\ref{thm:Mainthm} to obtain the value of $\omega(\mc{G})$ for three minor-closed classes of graphs. We do this by investigating classes of generalised triangles defined by properties of their $2$-cuts.

\begin{definition}\label{def:P1P2}
Let $G$ be a graph.
\begin{enumerate}
\item A $2$-cut $\{x,y\}$ of $G$ has property $P_1$ if for every $\{x,y\}$-bridge $B$, at least one of $x$ and $y$ has degree 1 in $B$.
\item A $2$-cut $\{x,y\}$ of $G$ has property $P_2$ if at least one $\{x,y\}$-bridge is trivial.
\end{enumerate}
\end{definition}

For $i \in \{1,2\}$, define $\K_i$ to be the family of generalised triangles satisfying property $P_i$ at every $2$-cut. 

\begin{lemma}\label{lem:K1K2downwardsClosed}
$\K_1$ and $\K_2$ are downward-closed subsets of $\K$.
\end{lemma}

\begin{proof}[Sketch of proof]
Let $i \in \{1,2\}$. It suffices to show that if $G \in \K_i$ and $G$ is formed from $G'$ by a single double subdivision, then $G' \in \K_i$. The contrapositive of this statement is much easier to see. Indeed if $G' \not\in \K_i$, then there is some $2$-cut $\{x,y\}$ of $G'$ which does not satisfy property $P_i$. The same vertices form a $2$-cut of $G$ which does not satisfy property $P_i$.
\end{proof}


\begin{subsection}{The family $\K_1$}

The aim of this section is to prove Theorem~\ref{thm:megaThm}(i). To do this we must show that $\omega(\K_1) = 5/4$, and that $\{H_0\}$ is the family of forbidden minors which characterises the class $\K_1$ within $\K$. The result then follows from Theorem~\ref{thm:Mainthm}.
 
\begin{lemma}\label{lem:w(K1)>=5/4}
If $G \in \K_1$, then $(-1)^{|V(G)|}P(G,t)>0$ for $t \in (1, 5/4]$.
\end{lemma}

The proof of this lemma is simple but fairly lengthy and can be found in Section~\ref{sec:proofK_1}. The idea is to prove several inequalities simultaneously by induction, one of which is the statement above. 

\begin{lemma}\label{lem:w(K1)=5/4}
$\omega(\K_1) = 5/4$.
\end{lemma}
\begin{proof}
Let $J_0 = K_3$ and $x \in V(J_0)$. For $i \in \mathbb{N}$, let $J_i$ be obtained from $J_{i-1}$ by applying the double subdivision operation to each edge of $J_{i-1}$ incident with $x$. Dong and Jackson~\cite{Nearly3Conn} say the graphs in this sequence have chromatic roots converging to $5/4$ from above. We shall show that $J_i \in \K_1$ for each $i \in \mathbb{N}_0$. It then follows that $\omega(\K_1) \leq 5/4$, which together with Lemma~\ref{lem:w(K1)>=5/4} implies that $\omega(\K_1)=5/4$.

Let $i \in \mathbb{N}_0$ and note that, by construction, every $2$-cut of $J_i$ contains the vertex $x$. Consider a $2$-cut $\{x,y\}$ and let $B$ be an $\{x,y\}$-bridge of $J_i$. Since $J_i$ is a generalised triangle, Proposition~\ref{prop:Properties1a}(ii) gives that $B$ has a cut-vertex $z$ which separates $x$ from $y$. If $y$ has degree at least $2$ in $B$, then $\{y,z\}$ is a $2$-cut of $G$, contradicting the fact that each $2$-cut contains $x$. Thus, for each $2$-cut $\{x,y\}$, the vertex $y$ has degree 1 in each $\{x,y\}$-bridge. We conclude that each $2$-cut has property $P_1$, so $J_i \in \K_1$ as desired.
\end{proof}

Recall that $H_0$ is the graph depicted in Figure~\ref{fig:graphs}. It is formed from $K_3$ by applying the double subdivision operation to each edge.

\begin{lemma}\label{lem:K1_iff_H0-free}
If $G \in \mc{K}$, then $G \in \K_1$ if and only if $G$ is $H_0$-minor-free.
\end{lemma}

\begin{proof}
By Lemma~\ref{lem:K1K2downwardsClosed}, $\K_1$ is a downward-closed subset of $\K$. Thus, by Lemma~\ref{lem:ForbMinor}, we need only determine that $\{H_0\}$ is the subset $\mc{F}$ of minimal elements of $\K \setminus \K_1$. So suppose $H \in \mc{F}$. Since $H\not\in \K_1$, there is some $2$-cut $\{x,y\}$ with $\{x,y\}$-bridges $B_1, B_2$ and $B_3$ such that both $x$ and $y$ have degree at least $2$ in $B_1$ say. We claim that the bridges $B_2$ and $B_3$ are trivial, so suppose for a contradiction that $B_2$, say, is not. By Proposition~\ref{prop:properties2}(ii), we can find a $2$-cut $\{u,v\} \subseteq V(B_2)$, such that the two $\{u,v\}$-bridges contained in $B_2$ are trivial. By Proposition~\ref{prop:properties2}(i), replacing these two bridges by a single edge $uv$ yields a generalised triangle $H'$, and $H$ can be obtained from $H'$ by a double subdivision. Because of $B_1$, the $2$-cut $\{x,y\}$ also does not satisfy $P_1$ in $H'$, so $H' \in \K \setminus \K_1$, which contradicts the minimality of $H$. Thus $B_2$ and $B_3$ are trivial.

Now $B_1$ has more than three vertices and as such has a cut-vertex $z$. Since both $x$ and $y$ have degree at least $2$ in $B_1$, both $\{x,z\}$ and $\{y,z\}$ are $2$-cuts of $H$. Let $X_1, X_2$ be the $\{x,z\}$-bridges not containing $y$, and let $Y_1, Y_2$ be the $\{y,z\}$-bridges not containing $x$. Since $H$ is minimal, the same reasoning as above implies that $X_1, X_2, Y_1, Y_2$ are trivial bridges. We conclude that $H = H_0$ and thus $\mc{F}= \{H_0\}$.
\end{proof}

Let $\mc{G}_1$ be the class of graphs such that some vertex is contained in every $2$-cut. In~\cite{Nearly3Conn} Dong and Jackson conjecture that $\omega(\mc{G}_1) = 5/4$. This conjecture is important, since $\mc{G}_1$ contains the class of $3$-connected graphs, and so a positive solution would give a lower bound on the non-trivial roots of $3$-connected graphs. While it can be shown that $\mc{G}_1\cap \K \subset \K_1$, this does not prove the conjecture since it is not known if $\omega(\mc{G}_1) = \omega(\mc{G}_1\cap \K)$. In particular $\mc{G}_1$ is not minor-closed so Theorem~\ref{thm:DongMinor} does not apply.

The fact that $\mc{G}_1\cap \K$ is not the largest class of generalised triangles $\K'$ such that $\omega(\K') = 5/4$ suggests that a well chosen weaker property could be used to make progress on Dong and Jackson's conjecture.

\begin{problem}
Find a class of graphs $\mc{G}$ such that $\mc{G}_1\subseteq \mc{G}$, $\omega(\mc{G}) = \omega(\mc{G}\cap \K)$ and $\mc{G}\cap \K = \K_1$.
\end{problem}

\end{subsection}


\begin{subsection}{The family $\K_2$}

In this section we prove Theorem~\ref{thm:megaThm}(ii). To do this we must show that $\omega(\K_2) = q$, where $q\approx 1.225$ is the unique real root of the polynomial $t^4-4t^3+4t^2-4t+4$ in the interval $(1,2)$. We must also show that $\{H_1, H_2\}$ is the family of forbidden minors which characterises the class $\K_2$ within $\K$. The result then follows from Theorem~\ref{thm:Mainthm}.

\begin{lemma}\label{lem:w(K2)>=q}
If $G \in \K_2$, then $(-1)^{|V(G)|}P(G,t)>0$ for $t \in (1,q]$.
\end{lemma}

The proof of this lemma is also simple but fairly lengthy and can be found in Section~\ref{sec:proofK_2}. The idea is the same as that of Lemma~\ref{lem:w(K1)>=5/4}.

\begin{lemma}\label{lem:w(K2)=q}
$\omega(\K_2) = q$.
\end{lemma}

\begin{proof}
Define $J_0 = K_3$ and consider an embedding of $J_0$ in the plane. For $i \in \mathbb{N}$, let $J_i$ be formed from $J_{i-1}$ by applying the double subdivision operation to each edge of $J_{i-1}$ on the outer face. In~\cite{Nearly3Conn}, Dong and Jackson say this sequence of graphs has chromatic roots converging to $q$ from above. The following claim implies that $J_i \in \K_2$ for $i \in \mathbb{N}_0$. It then follows that $\omega(\K_2) \leq q$, which together with Lemma~\ref{lem:w(K2)>=q} gives that $\omega(\K_2) = q$.

\textbf{Claim: For $i \in \mathbb{N}_0$, every $2$-cut of $J_i$ has a trivial bridge which lies inside the outer cycle of $J_i$.}\\ We prove the claim by induction on $i$. The result holds vacuously for $i=0$ and is easily checked for $i=1$, so suppose the result is true for $k \in \mathbb{N}_0$. For the induction step, note that the $2$-cuts of $J_{k+1}$ consist of the $2$-cuts of $J_k$, and $\{u,v\}$ for every edge $uv$ of the outer cycle of $J_k$. If $\{x,y\}$ is a $2$-cut of $J_k$, then by the induction hypothesis there is a trivial $\{x,y\}$-bridge which lies inside the outer cycle of $J_k$. This bridge is left unchanged in $J_{k+1}$ so $\{x,y\}$ still satisfies the hypothesis. Alternatively, if $uv$ is an edge of the outer cycle of $J_k$, then in $J_{k+1}$, the edge $uv$ is replaced with two trivial $\{u,v\}$-bridges. One of these bridges forms part of the outer cycle of $J_{k+1}$, whilst the other bridge lies inside the new outer cycle as required.
\end{proof}

Let $H = K_{2,3}$ with vertex partition $\{\{x,y\},\{u,v,w\}\}$. The graph $H_1$ is formed from $H$ by applying the double subdivision operation to each edge adjacent to $x$. The graph $H_2$ is formed from $H$ by applying the double subdivision operation to the edges $xu, xv$ and $yw$, see Figure~\ref{fig:graphs}.

\begin{lemma}\label{lem:K2_iff_H1H2-free}
If $G \in \K$, then $G \in \K_2$ if and only if $G$ is $\{H_1, H_2\}$-minor-free.
\end{lemma}

\begin{proof}
By Lemma~\ref{lem:K1K2downwardsClosed}, $\K_2$ is a downward-closed subset of $\K$. Thus by Lemma~\ref{lem:ForbMinor} we need only determine that $\mc{F}(\K_2) = \{H_1, H_2\}$. To this end, let $G$ be a minimal element of $\K \setminus \K_2$. Since $G \not\in \K_2$, there is a $2$-cut $\{x,y\}$ of $G$ with $\{x,y\}$-bridges $B_1, B_2$, and $B_3$, none of which is trivial. Suppose that, for some $i \in \{1,2,3\}$, the bridge $B_i$ has more than five vertices. By Proposition~\ref{prop:properties2}(ii), there is a $2$-cut $\{u,v\} \subseteq V(B_i)$ such that the two $\{u,v\}$-bridges contained in $B_i$ are trivial. Let $B'$ denote the third $\{u,v\}$-bridge, and let $G'$ be the graph formed from $B'$ by adding the edge $uv$. Note that $G' \in \K$ by Proposition~\ref{prop:properties2}(i). Let $B'_i$ be the $\{x,y\}$-bridge of $G'$ corresponding to $B_i$ in $G$. The other $\{x,y\}$-bridges of $G$ are left unchanged in $G'$. Since $B_i$ has more than $5$ vertices, $B'_i$ is not trivial. Thus $G' \in \K \setminus \K_2$. Since $G$ can be formed from $G'$ by applying the double subdivision operation to $uv$, this contradicts the minimality of $G$. Therefore each of $B_1$, $B_2$ and $B_3$ has precisely five vertices and so is formed from a trivial $\{x,y\}$-bridge by precisely one double subdivision. We conclude that $G \in \{H_1, H_2\}$.
\end{proof}

Let $\mc{G}_2$ be the class of $2$-connected plane graphs such that every $2$-cut is contained in the outer-cycle. In~\cite{Nearly3Conn}, Dong and Jackson conjecture that $\omega(\mc{G}_2) = q$. Once again, this is an important conjecture since $\mc{G}_2$ contains the class of $3$-connected planar graphs. Whilst it can be shown that $\mc{G}_2 \cap \K \subset \K_2$, this does not prove the conjecture since it is not known if $\omega(\mc{G}_2) = \omega(\mc{G}_2\cap \K)$. In particular $\mc{G}_2$ is not minor-closed so Theorem~\ref{thm:DongMinor} does not apply. 

Again, the fact that $\mc{G}_2\cap \K$ is not the largest class of generalised triangles $\K'$ such that $\omega(\K') = q$ suggests that a well chosen weaker property could be used to make progress on Dong and Jackson's conjecture.

\begin{problem}
Find a class of graphs $\mc{G}$ such that $\mc{G}_2\subseteq \mc{G}$, $\omega(\mc{G}) = \omega(\mc{G}\cap \K)$ and $\mc{G}\cap \K = \K_2$.
\end{problem}

\end{subsection}

\begin{subsection}{The family $\K_1 \cap \K_2$}

In this section we show that $\omega(\K_1\cap \K_2) = t_0$, where $t_0\approx 1.296$ is the unique real root of the polynomial $t^3-2 t^2+4 t-4$. Theorem~\ref{thm:megaThm}(iii) then follows from Lemma~\ref{lem:K1_iff_H0-free}, Lemma~\ref{lem:K2_iff_H1H2-free} and Theorem~\ref{thm:Mainthm}.

We require the following proposition regarding an operation called a \emph{Whitney $2$-switch}. 
\begin{proposition}\label{prop:WhitneySwitch}
Let $G$ be a graph and $\{x,y\}$ be a $2$-cut of $G$. Let $C$ denote a component of $G - \{x,y\}$. Define $G'$ to be the graph obtained from the disjoint union of $G-C$ and $C$ by adding for all $z \in V(C)$ the edge $xz$ (respectively $yz$) if and only if $yz$ (respectively $xz$) is an edge of $G$. Then we have $P(G,t) = P(G',t).$
\end{proposition}

\begin{proof}[Sketch of proof]
If $xy \in E(G)$, then apply Proposition~\ref{prop:factor} to the $2$-cut $\{x,y\}$ in both $G$ and $G'$. The resulting expressions for $P(G,t)$ and $P(G',t)$ are the same. If $xy \not\in E(G)$, then first apply Proposition~\ref{prop:delcont}(ii) to the pair $\{x,y\}$ in both $G$ and $G'$. Next, apply Proposition~\ref{prop:factor} to $G+xy, G/xy, G'+xy$ and $G'/xy$. Again the resulting expressions for $P(G,t)$ and $P(G',t)$ are the same.
\end{proof}

Let $\mc{H}$ denote the family of graphs which have a Hamiltonian path. To prove Theorem~\ref{thm:THomassenHamPath}, Thomassen implicitly proved the following lemma.

\begin{lemma}\emph{\cite{ThomassenHamPath}}\label{lem:thomassenHamPathGenTri}
$\omega(\mc{H}) = \omega(\mc{H}\cap \K)$.
\end{lemma}

The next two lemmas show that $\{P(H,t) : H \in \mc{H} \cap \K\} = \{P(G,t) : G \in \K_1 \cap \K_2\}$, whence $\omega(\mc{H}\cap \K) = \omega(\K_1 \cap \K_2)$. The fact that $\omega(\K_1 \cap \K_2) = t_0$ then follows from Theorem~\ref{thm:THomassenHamPath} and Lemma~\ref{lem:thomassenHamPathGenTri}.

\begin{figure}
\centering
\includegraphics[scale=0.55]{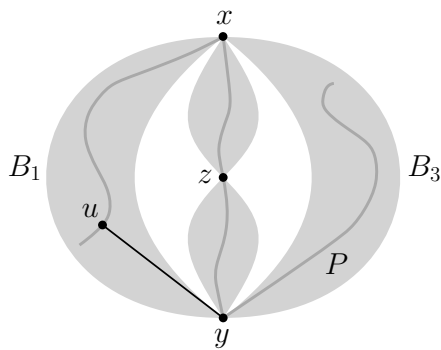}
\caption{The structure of a graph in $\mc{H} \cap \K$.}\label{fig:hamPath}
\end{figure}
\begin{lemma}\label{lem:TsubsetK1K2}
$\mc{H} \cap \K \subseteq \K_1 \cap \K_2$.
\end{lemma}

\begin{proof}
Suppose $G \in \mc{H}\cap \K$ and let $P$ denote the Hamiltonian path of $G$. If $G = K_3$ we are done so we may assume that $G$ contains a $2$-cut. Let $\{x,y\}$ be an arbitrary $2$-cut of $G$. We shall show that $\{x,y\}$ has properties $P_1$ and $P_2$. Since $G$ is a generalised triangle, there are precisely three $\{x,y\}$-bridges $B_1, B_2$ and $B_3$, none of which is $2$-connected. Without loss of generality, assume that $P$ begins in $B_1$, visits $x$ before $y$, and ends in $B_3$, see Figure~\ref{fig:hamPath}. Finally, let $z$ be the cut-vertex of $B_2$. 

\textbf{Claim: $\{x,z\}$ and $\{y,z\}$ are not $2$-cuts of $G$.}

Suppose for a contradiction that $\{x,z\}$ is a $2$-cut. Let $uy$ be an edge from $y$ to a~vertex $u \in V(B_1)$. The graph $G' = P \cup uy$ is a spanning subgraph of $G$ such that $G'-\{x,z\}$ has at most $2$-components, see Figure~\ref{fig:hamPath}. It follows that $G-\{x,z\}$ can have at most two components, which contradicts Proposition~\ref{prop:Properties1a}(ii). The situation is symmetric so the same proof shows that $\{y,z\}$ is not a $2$-cut of $G$. This completes the proof of the claim.

Since $z$ is a cut-vertex of $B_2$, the claim implies that $B_2$ is the path $xzy$, a trivial $\{x,y\}$-bridge. Thus $\{x,y\}$ has property $P_2$.

We now show that $\{x,y\}$ satisfies $P_1$, so suppose for a contradiction that this is not the case. Thus, there is an $\{x,y\}$-bridge $B_1$, say, such that both $x$ and $y$ have degree at least $2$ in $B_1$. Since $G$ is a generalised triangle, $B_1$ has a cut vertex $w$ which separates $x$ from $y$. Now, because $x$ and $y$ have degree at least $2$ in $B_1$, both $\{x,w\}$ and $\{y,w\}$ are $2$-cuts of $G$, and since $G$ is a generalised triangle, both $\{x,w\}$ and $\{y,w\}$ have precisely three bridges. Two of the $\{x,w\}$-bridges lie in $B_1$, and the same is true for $\{y,w\}$. Therefore, $G - \{x,y,w\}$ has six components. However, since $G$ has a Hamiltonian path, deleting $r$ vertices from $G$ can leave at most $r+1$ components. This gives the required contradiction.
\end{proof}

\begin{lemma}\label{lem:K1K2equivT}
If $G \in \K_1 \cap \K_2$, then there is $H \in \mc{H} \cap \K$ such that $P(G,t) = P(H, t)$.
\end{lemma}

\begin{proof}
Let $G \in \K_1 \cap \K_2$. By the characterisation of generalised triangles in Proposition~\ref{prop:Properties1a}, it is easy to see that $\mc{K}$ is invariant under Whitney $2$-switches. Thus we need only prove that $G$ can be transformed into a graph with a Hamiltonian path by a sequence of Whitney $2$-switches. The result clearly holds if $G = K_3$ so we may suppose that $|V(G)|>3$. We first prove the following claim.

\textbf{Claim: Let $\{x,y\}$ be a $2$-cut of $G$ and $B$ be an $\{x,y\}$-bridge of $G$. If $y$ has degree $1$ in $B$, then there is a sequence of Whitney $2$-switches in $G$ such that in the resulting graph, the $\{x,y\}$-bridge corresponding to $B$ contains a path $P_B$ starting at $x$, and such that $V(P_B) = V(B) \setminus \{y\}$.}

\begin{figure}
\centering
\includegraphics[scale=.55]{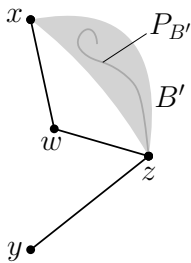}
\caption{The $\{x,y\}$-bridge $B$ in the proof of Lemma~\ref{lem:K1K2equivT}.}\label{fig:induction}

\end{figure}
We proceed by induction on $|V(B)|$. If $|V(B)| = 3$, then $B$ is trivial and the result is clear. Thus we may suppose $|V(B)|>3$. Let $z$ be the neighbour of $y$ in $B$. Since $|V(B)| > 3$, we have that $\{x,z\}$ is a $2$-cut of $G$ with precisely three $\{x,z\}$-bridges, two of which, $B'$ and $B''$, are contained in $B$. Since $G \in \K_2$, one of $B'$ and $B''$ is trivial, say $B''$ is the path $xwz$, see Figure~\ref{fig:induction}. Furthermore, since $G \in \K_1$, at least one of $x$ and $z$ has degree $1$ in $B'$. If necessary, we perform a Whitney $2$-switch of $B'$ about $\{x,z\}$ so that this vertex is $x$, and call the resulting graph $G'$. Now $|V(B')| < |V(B)|$, so by induction, we have that there is a sequence of Whitney $2$-switches in $G'$ such that in the resulting graph $G''$, the bridge corresponding to $B'$ contains a path $P_{B'}$, starting at $z$, and such that $V(P_{B'}) = V(B') \setminus \{x\}$. Now $xwz\cup P_{B'}$ is the desired path $P_B$ in $G''$, see Figure~\ref{fig:induction}. This completes the proof of the claim.

To prove the lemma, let $\{x,y\}$ be a $2$-cut of $G$ such that two of the $\{x,y\}$-bridges $B_1$ and $B_2$ are trivial with vertex-sets $\{x,y,u\}$ and $\{x,y,v\}$ respectively. Such a $2$-cut can easily be found by considering the construction of $G$ from a triangle by double subdivisions. Let $B$ be the remaining $\{x,y\}$-bridge. By the claim above, there is a sequence of Whitney $2$-switches in $G$ such that the resulting graph has a path $P_B$ starting at $x$ and covering all vertices of $B$ except for $y$. In the resulting graph, $uyvx\cup P_B$ is a Hamiltonian path.
\end{proof}

We remark that in~\cite{3LST}, the present author proved an analogue of Thomassen's result for a slightly more general class of graphs.

\begin{theorem}\label{thm:3LST}
If $\mathcal{H}'$ is the class of graphs containing a spanning tree with at most three leaves, then $\omega(\mathcal{H}') = t_1$ where $t_1 \approx 1.290$ is the smallest real root of the polynomial $t^6-8 t^5+27 t^4-56 t^3+82 t^2-76 t+31$.
\end{theorem}
\end{subsection}

To prove Theorem~\ref{thm:3LST} it was shown that $\omega(\mc{H}') = \omega(\mc{H}' \cap \K)$ and $\omega(\mc{H}' \cap \K) = t_1$. Using the method detailed in this section, one can similarly find a minor-closed family $\mc{G}$, such that $\omega(\mc{G}\cap \K) = \omega(\mc{H}') = t_1$. For details we refer the reader to~\cite{thesis}.

\end{section}


\begin{section}{Proofs of the lemmas}\label{sec:proofLemmas}

In this section we prove Lemma~\ref{lem:w(K1)>=5/4} and Lemma~\ref{lem:w(K2)>=q}. The proofs are similar to Jackson's proof in~\cite{Jackson32/27} of the result that $\omega(\mc{\K}) \geq 32/27$, except that the additional structure of the classes $\K_1$ and $\K_2$ allows us to make some savings and get a larger value for $\omega(\K_1)$ and $\omega(\K_2)$. The proofs are fairly long, but nevertheless rely only on the basic identities introduced in Proposition~\ref{prop:delcont} and Proposition~\ref{prop:factor}.

For $t \in (-\infty, 32/27]$, the sign of the chromatic polynomial of a $2$-connected graph is completely determined and depends on the number of vertices, see~\cite{Jackson32/27}. For this reason we will find it useful to work with the function $Q(G,t) = (-1)^{|V(G)|}P(G,t)$. We shall make repeated use of equalities (i) and (ii) in Proposition~\ref{prop:delcont} which will be referred to as \emph{deletion-contraction} and \emph{addition-contraction} respectively. If $xy$ is an edge of a graph $G$, then $G-xy$ denotes the graph formed by deleting $xy$. We denote by $G/xy$ the graph formed from $G$ by identifying the vertices $x$ and $y$, and deleting all loops and multiple edges created. In this case $xy$ need not be an edge of the graph.

\begin{proposition}\label{prop:delcont}
Let $G$ be a graph and $x,y \in V(G)$.
\begin{enumerate}[(i)]
\item If $xy \in E(G)$, then $Q(G,t) = Q(G-xy,t)+Q(G/xy,t)$.
\item If $xy \not\in E(G)$, then $Q(G,t) = Q(G+xy,t)-Q(G/xy,t)$.
\end{enumerate}
\end{proposition}

The following proposition will be used frequently for $r=1$ and $r=2$.

\begin{proposition}\label{prop:factor}
If $G$ is a graph such that $G = G_1 \cup G_2$ and $G_1 \cap G_2 = K_r$ for some $r \in \mathbb{N}$, then $$Q(G,t) = \frac{Q(G_1,t)Q(G_2,t)}{Q(K_r,t)} = \frac{Q(G_1,t)Q(G_2,t)}{(-1)^r t(t-1)\cdots(t-r+1)}.$$
\end{proposition}

For use in the following proofs, we define a \emph{generalised edge} to be either $K_2$ or any graph obtained from $K_2$ by a sequence of double subdivisions. When $V(K_2) = \{u,v\}$ we shall refer to a generalised edge obtained from this $K_2$ as a \emph{generalised $uv$-edge}. Let $G$ be a generalised $uv$-edge with $|V(G)| \geq 4$, and let $B_1$ and $B_2$ be the $\{u,v\}$-bridges of $G$. Recall the properties of a $2$-cut defined in Definition~\ref{def:P1P2}. For $i \in \{1,2\}$, we say $G$ has property $P_i$ if every $2$-cut $\{x,y\}$ such that $\{x,y\} \subseteq V(B_j)$ for some $j \in \{1,2\}$ has property $P_i$.

The following extra property of generalised triangles will be useful in this analysis.
\begin{proposition}\emph{\cite{DongKohMethod}}\label{prop:Properties1b}
If $G$ is a generalised triangle, then for every edge $uv$, $G-uv = G_1 \cup G_2$, where $G_1$ is a generalised $uz$-edge, $G_2$ is a generalised $vz$-edge, and  $G_1 \cap G_2 = \{z\}$.
\end{proposition}

\begin{subsection}{Proof of Lemma~\ref{lem:w(K1)>=5/4}}\label{sec:proofK_1}
Lemma~\ref{lem:w(K1)>=5/4} is statement (e) in the following lemma.
\begin{lemma}\label{lem:K_1}
Let $G$ be a graph and let $t \in (1,5/4]$.
\begin{enumerate}[(a)]
\item If $G \in \K_1$ and $v$ is a vertex of degree $2$ in $G$ with neighbours $u$ and $w$, then $Q(G,t) \geq \frac{1}{2}Q(G / uv,t)$.
\item If $G$ is a generalised $uw$-edge with property $P_1$ and $|V(G)|\geq 4$, then $Q(G + uw,t) \geq \frac{1}{2} Q(G,t)$.
\item If $G \in \K_1$ and $v$ is a vertex of degree $2$ in $G$ with neighbours $u$ and $w$, then $Q(G / uv,t) >0$.
\item If $G$ is a generalised $uw$-edge with property $P_1$ then $Q(G,t)>0$.
\item If $G \in \K_1$, then $Q(G,t)>0$.
\end{enumerate}
\end{lemma}

\begin{proof}

We prove the results simultaneously by induction on $|V(G)|$. If $|V(G)|\leq 4$ then either $G = K_3$ if $G \in \K_1$ or $G = C_4$ if $G$ is a generalised edge with property $P_1$. Thus (c), (d) and (e) are easily verified. Part (a) also holds since $Q(K_3,t) = (2-t)Q(K_2,t) \geq  \frac{3}{4}Q(K_2,t) >  \frac{1}{2} Q(K_2,t)$. Finally (b) holds when $G = C_4$ since $Q(C_4 +uw,t) - \frac{1}{2}Q(C_4, t) = \frac{1}{2} t(t-1)((t-2)^2 -(t-1))>0.$  Thus we may suppose $|V(G)| > 4$ and that (a) to (e) hold for all graphs with fewer vertices. 

\begin{enumerate}[(a)]
\item  Set $H = G -v$. Note that $H$ is a generalised $uw$-edge with property $P_1$ and $G/uv = H +uw$. By deletion-contraction and Proposition~\ref{prop:factor} we have
\begin{align*}
Q(G,t) &= Q(G-uv,t) + Q(G/uv,t) \\
&= (1-t) Q(H,t) + Q(H + uw,t).\numberthis \label{eq:K1a1}
\end{align*}

By the induction hypothesis of (d) on $H$, we have $Q(H,t)>0$. Furthermore, by the induction hypothesis of (b) on $H$, we have $Q(H+uw,t)\geq \tfrac{1}{2}Q(H,t)$. Using the fact that $t \in (1, 5/4]$, equation \eqref{eq:K1a1} becomes
\begin{align*}
Q(G,t) & \geq 2 (1-t) Q(H +uw,t) + Q(H + uw,t) \\
& = (3-2t)Q(H+uw,t)\\
& \geq \tfrac{1}{2}Q(H+uw,t)\\
& = \tfrac{1}{2}Q(G/uv,t).
\end{align*}

\item Let $s = Q(G + uw,t) - \frac{1}{2}Q(G,t)$. Also let $H_1$ and $H_2$ be the $\{u,w\}$-bridges of the graph $G + uw$ and note that $H_1, H_2 \in \K_1$. By addition-contraction and Proposition~\ref{prop:factor},
\begin{align*}
s &= Q(G + uw,t) - \tfrac{1}{2} [ Q(G +uw,t) - Q(G/uw,t) ] \\
&= \tfrac{1}{2} Q(G + uw,t) + \tfrac{1}{2}Q(G/uw,t) \\
& = \tfrac{1}{2}t^{-1}(t-1)^{-1}Q(H_1,t)Q(H_{2},t) - \tfrac{1}{2}t^{-1}Q(H_1/uw,t)Q(H_{2}/uw,t).
\end{align*}

By the induction hypotheses of (c) and (e), we have $Q(H_i/uw,t)>0$ and $Q(H_i,t)>0$ for $i \in \{1,2\}$. Since the $2$-cut $\{u,w\}$ of $G$ has property $P_1$, in each of $H_1$ and $H_2$ at least one of the vertices $u$ and $w$ has degree $2$. Therefore, for $i \in \{1,2\}$, the induction hypothesis of (a) on the edge $uw$ of $H_i$ implies that $Q(H_i,t) \geq \tfrac{1}{2}Q(H_i/uw,t)$. Now since $t \in (1, 5/4]$, 
\begin{align*}
2s t(t-1) & = Q(H_1,t)Q(H_2,t) - (t-1)Q(H_1/uw,t)Q(H_2/uw,t)\\
& \geq Q(H_1/uw,t)Q(H_2/uw,t)[(\tfrac{1}{2})^2 - \tfrac{1}{4} ]=  0.
\end{align*}

\item Since $v$ has degree $2$, the set $\{u,w\}$ is a $2$-cut of $G/uv$ and $uw \in E(G/uv)$. Thus the $\{u,w\}$-bridges $H_1$ and $H_2$ of $G/uv$ are members of $\K_1$ and so $Q(H_i, t)>0$ for $i \in \{1,2\}$ by the induction hypothesis of (e). Finally, since $H_1$ and $H_2$ intersect in a complete subgraph, $$Q(G/uv,t) = t^{-1}(t-1)^{-1}Q(H_1,t)Q(H_2,t)>0.$$

\item Let $H_1$ and $H_2$ be the $uw$-bridges of $G+uw$ and note that $H_1,H_2 \in \K_1$. By the induction hypothesis of (e), we have $Q(H_i,t) >0$ for $i \in \{1,2\}$. Since $G$ is a generalised edge with property $P_1$, the $2$-cut $\{u,w\}$ of $G$ has property $P_1$ and so in each of $H_1, H_2$, at least one of $u$ or $w$ has degree $2$. Thus, by the induction hypothesis of (c), we have $Q(H_i/uw,t)>0$ for $i \in \{1,2\}$. 
Now addition-contraction and Proposition~\ref{prop:factor} give,
\begin{align*}
Q(G,t) &= Q(G+uw,t) - Q(G /uw,t)\\
&= t^{-1}(t-1)^{-1}Q(H_1,t)Q(H_2,t) + t^{-1}Q(H_1/uw,t)Q(H_2/uw,t) >0.
\end{align*}

\item Firstly, note that (a) and (c) have now been proven for a graph on $|V(G)|$ vertices. Let $v$ be a vertex of degree $2$ with neighbours $u$ and $w$. By (a), $Q(G,t) \geq \frac{1}{2} Q(G/uv,t)$, and by (c), $Q(G/uv,t)>0$. Therefore $Q(G,t)>0$.\end{enumerate}\end{proof}
\end{subsection}


\begin{subsection}{Proof of Lemma~\ref{lem:w(K2)>=q}}\label{sec:proofK_2}
To prove Lemma~\ref{lem:w(K2)>=q} we require a few preliminary results. Recall that $q\approx 1.225$ is the unique real root of the polynomial $t^4-4t^3+4t^2-4t+4$ in $(1,2)$ and define the constants
\begin{align*}
\gamma &=\tfrac{1}{4}(q-2)(q^2-2q-2) \approx 0.571,\\
\alpha &= (1-\gamma)(2-q)(2-q-\gamma)^{-1} \approx 1.632,\\
\beta &= 1-\alpha^{-1} = \gamma(q-1)(1-\gamma)^{-1}(2-q)^{-1} \approx 0.387.
\end{align*}

\begin{lemma}\label{lem:AlphaBetaGamma}
For all $t \in (1, q]$ we have 
\begin{enumerate}[(i)]
\item $t(t-1)^{-1}\gamma^2-2\gamma+1\geq \alpha$
\item $(1-t)\gamma^{-1} +1 \geq \beta$
\item $(1-\gamma)(2-t)\beta-(t-1)\gamma \geq 0$.
\end{enumerate}
\end{lemma}

\begin{proof}
For $t \in (1,q]$, the left hand sides of the three inequalities are decreasing functions of $t$. Thus we need only verify them for $t = q$. Now (i) and (ii) can be verified by lengthy substitution using the expression for $\gamma$. Part (iii) follows immediately from the definition of $\beta$.
\end{proof}

The following useful reduction lemma is due to Jackson.

\begin{lemma}\emph{~\cite{Jackson32/27}}\label{lem:JacksonReduction}
Let $G$ be a $2$-connected graph and $\{u,v\}$ be a $2$-cut such that $uv$ is not an edge of $G$. If $G_1$ and $G_2$ are subgraphs of $G$ such that $G_1 \cup G_2 = G$, $V(G_1)\cap V(G_2) = \{u,v\}$, $|V(G_1)| \geq 3$ and $|V(G_2)|\geq 3$, then
\begin{multline*}
t(t-1)Q(G,t) = tQ(G_1+uv,t)Q(G_2+uv,t)\\ + (t-1)\left[Q(G_1,t)Q(G_2,t)-Q(G_1+uv,t)Q(G_2,t)-Q(G_1,t)Q(G_2+uv,t)\right].
\end{multline*}
\end{lemma}

Now Lemma~\ref{lem:w(K2)>=q} is statement (e) in the following result.

\begin{lemma}\label{lem:K_2analysis}
Let $G$ be a graph and let $t \in (1,q]$.
\begin{enumerate}[(a)]
\item Suppose $G = G_1 \cup G_2$ where $G_1$ and $G_2$ are generalised $uv$-edges such that $G_1 \cap G_2 = \{u,v\}$, $|V(G_1)|\geq 4$ and $|V(G_2)|\geq 4$. If $G_1$ and $G_2$ have property $P_2$, then $$Q(G,t) \geq \alpha t^{-1} Q(G_1,t)Q(G_2,t).$$
\item Suppose $G = G_1 \cup G_2 +uv$ where $G_1$ is a generalised $uw$-edge, $G_2$ is a generalised $vw$-edge and $G_1 \cap G_2 = \{w\}$.  If $G_1$ and $G_2$ have property $P_2$, then $$Q(G,t)\geq \beta Q(G/uv,t).$$
\item If $G$ is a generalised $uv$-edge with property $P_2$ and $|V(G)|\geq 4$, then $$Q(G+uv,t) \geq \gamma Q(G,t).$$
\item Suppose $G = G_1 \cup G_2$ where $G_1$ and $G_2$ are generalised $uw$-edges such that $G_1 \cap G_2 = \{u,w\}$. If $G_1$ and $G_2$ have property $P_2$, then $Q(G,t) > 0.$
\item If $G \in \K_2$ then $Q(G,t)>0$.
\item If $G$ is a generalised $uw$-edge with property $P_2$ then $Q(G,t)>0$.
\end{enumerate}
\end{lemma}

\begin{proof}
We prove the results simultaneously by induction on $|V(G)|$. If $|V(G)|\leq 4$ then either $G = K_3$ if $G \in \K_2$ or $G = C_4$ if $G$ is a generalised edge with property $P_2$. Thus (e) and (f) are easily verified. Part (b) also holds since $Q(K_3,t) = (2-t)Q(K_2,t) >  \frac{3}{4}Q(K_2,t) >  \beta Q(K_2)$. Part (c) holds when $G = C_4$ since $$Q(C_4 +uw,t) - \gamma Q(C_4, t) = t(t-1)((1-\gamma)(t-2)^2 -\gamma(t-1))>0.$$  Parts (a) and (d) are vacuously true, thus we may suppose $|V(G)| > 4$ and that (a) to (f) hold for all graphs with fewer vertices. 
\begin{enumerate}[(a)]
\item Applying Lemma~\ref{lem:JacksonReduction} to $G$ and rearranging, we have 
\begin{align*}
tQ(G,t) = &Q(G_1+uv,t)\left[ \tfrac{1}{2}t(t-1)^{-1}Q(G_2+uv,t) -Q(G_2,t)\right]\\
+ &Q(G_2+uv,t)\left[ \tfrac{1}{2}t(t-1)^{-1}Q(G_1+uv,t) -Q(G_1,t)\right]\\
+ &Q(G_1,t)Q(G_2,t). \numberthis \label{eq:K2a1}
\end{align*}
By the induction hypothesis of (c), we have $Q(G_i+uv,t) \geq \gamma Q(G_i,t)$ for $i \in \{1,2\}$. Also, by the induction hypothesis of (e), $Q(G_i,t)>0$ for $i \in \{1,2\}$. Substituting into~\eqref{eq:K2a1} and using Lemma~\ref{lem:AlphaBetaGamma}(i) now gives 
\begin{align*}
tQ(G,t) &\geq Q(G_1, t)Q(G_2,t)\left[t(t-1)^{-1}\gamma^2-2\gamma+1\right]\\
&\geq \alpha Q(G_1,t)Q(G_2,t).
\end{align*}

\item If one of $G_1$ and $G_2$ is a single edge, say $G_1 = uw$, then $G/uv$ = $G_2+vw$. By deletion-contraction on $uv$ and Proposition~\ref{prop:factor},
\begin{equation}\label{eq:K2b1}
Q(G,t) = (1-t)Q(G_2,t)+Q(G_2+vw,t).
\end{equation}
The induction hypothesis of (f) gives that $Q(G_2,t)>0$. Moreover the induction hypothesis of (c) on $G_2$ gives $Q(G_2+vw,t)\geq \gamma Q(G_2,t)$. Substituting into~\eqref{eq:K2b1} and using Lemma~\ref{lem:AlphaBetaGamma}(ii) we get 
\begin{equation}\label{eq:K2b2}
Q(G,t)\geq ((1-t)\gamma^{-1}+1)Q(G_2+vw) \geq \beta Q(G_2+vw,t) = \beta Q(G/uv,t).
\end{equation}

So suppose that both $G_1$ and $G_2$ have at least $4$ vertices. Deletion-contraction on $uv$ and Proposition~\ref{prop:factor} yield
\begin{equation}\label{eq:K2b3}
Q(G,t) = -t^{-1}Q(G_1,t)Q(G_2,t)+Q(G/uv,t).
\end{equation}
The induction hypothesis of (f) gives $Q(G_i,t)>0$ for $i \in \{1,2\}$, and the induction hypothesis of (a) gives $t^{-1}Q(G_1,t)Q(G_2,t)\leq \alpha^{-1}Q(G/uv,t)$. Substituting into~\eqref{eq:K2b3} we have
\begin{equation}\label{eq:K2b4}
Q(G,t)\geq (1-\alpha^{-1})Q(G/uv,t) = \beta Q(G/uv,t).
\end{equation}

\item Let $s = Q(G+uv,t) - \gamma Q(G,t)$. Since the $2$-cut $\{u,v\}$ of $G$ has property $P_2$, one $\{u,v\}$-bridge of $G$ is trivial. Let $H$ be the other $\{u,v\}$-bridge of $G$ and notice that $H+uv \in \K_2$. By addition-contraction on $G$, and using Proposition~\ref{prop:factor} we get
\begin{align*}
s & = Q(G+uv,t) - \gamma [Q(G+uv,t)-Q(G/uv, t)]\\ \numberthis \label{eq:K2c1}
&= (1-\gamma)Q(G+uv,t)+\gamma Q(G/uv,t)\\
&= (1-\gamma)(2-t)Q(H+uv,t) - \gamma (t-1)Q(H/uv,t).
\end{align*}
Note that $H = H_1 \cup H_2$ where $H_1$ is a generalised $uw$-edge with property $P_2$, $H_2$ is a generalised $vw$-edge with property $P_2$, and $H_1 \cap H_2 = \{w\}$. Thus by the induction hypothesis of (d), $Q(H/uv,t)>0$. Now by the induction hypothesis of (b), we have $Q(H+uv,t)\geq \beta Q(H/uv,t)$. Substituting into~\eqref{eq:K2c1} and using Lemma~\ref{lem:AlphaBetaGamma}(iii) gives $$s \geq [(1-\gamma)(2-t)\beta-\gamma(t-1)]Q(H/uv,t)\geq 0.$$

\item  If one of $G_1, G_2$ is a single edge, then $G$ is either a single edge, or $G = H_1 \cup H_2$, where $H_1,H_2 \in \K_2$, and $H_1 \cap H_2$ is the edge $uw$. By the induction hypothesis of (e) and Proposition~\ref{prop:factor}, we conclude that $Q(G,t)>0$. So suppose both $G_1$ and $G_2$ have at least $4$ vertices. By (a), which has now been proven for a graph on $|V(G)|$ vertices, we conclude that $Q(G,t) \geq \alpha t^{-1} Q(G_1,t)Q(G_2,t)$. By the induction hypothesis of (f), $Q(G_i,t)>0$ for $i \in \{1,2\}$. Therefore $Q(G,t)>0$. 

\item Let $\{u,w\}$ be a $2$-cut of $G$ so that two of the $\{u,w\}$-bridges are trivial. Such a $2$-cut is easily found by considering the construction of $G$ from $K_3$ by the double subdivision operation. Let $v$ be a vertex of degree $2$ in $G$ with neighbours $u$ and $w$. By Proposition~\ref{prop:Properties1b}, we may write $G = G_1 \cup G_2 + uv$ where $G_1$ is a generalised $vw$-edge, $G_2$ is a generalised $uw$-edge, and $G_1 \cap G_2 = \{w\}$. By the choice of $\{u,w\}$ we have in particular that $G_1$ is the edge $vw$, and $G_2$ is a generalised $uw$-edge with property $P_2$.

Now we may apply (b) to deduce $Q(G,t)\geq \beta Q(G/uv,t)$. Note that $G/uv = H_1 \cup H_2$ where $H_1, H_2 \in \K_2$ and $H_1\cap H_2$ is the edge $uw$. By the induction hypothesis of (e) we have that $Q(H_i,t)>0$ for $i \in \{1,2\}$. Now finally Proposition~\ref{prop:factor} gives $Q(G/uv,t) = t^{-1}(t-1)^{-1}Q(H_1,t)Q(H_2,t)>0$, whence $Q(G,t)>0$.

\item Let $v$ be a vertex of degree $2$ with neighbours $u$ and $w$. Let $H = G-v$ and $z$ be a cut-vertex of  $H$. Note that $H = H_1 \cup H_2$ where $H_1$ is a generalised $uz$-edge with property $P_2$, $H_2$ is a generalised $wz$-edge with property $P_2$, and $H_1\cap H_2 = \{z\}$. Note also that this implies $H+uw \in \K_2$. By addition-contraction and Proposition~\ref{prop:factor}, 
\begin{align*}
Q(G,t) &= Q(G+uw,t)-Q(G/uw,t)\\
& = (2-t)Q(H+uw,t)+(t-1)Q(H/uw,t).
\end{align*}
By the induction hypothesis of (e), we have $Q(H+uw,t)>0$. If one of $H_1$ or $H_2$ is a single edge, then $H/uw$ is either a single edge or an element of $\K_2$. In either case $Q(H/uw,t)>0$. Thus we may suppose both $H_1$ and $H_2$ have at least $4$ vertices. By the induction hypothesis of (f), $Q(H_i,t)>0$ for $i \in \{1,2\}$. Now we apply the induction hypothesis of (a) to get $$Q(H/uw,t)\geq \alpha t^{-1}Q(H_1,t)Q(H_2,t)>0.$$
\end{enumerate}
\end{proof}
\end{subsection}
\end{section}
\bibliographystyle{plain}
\bibliography{lit}

\end{document}